\documentclass[10pt,a4paper,english]{amsart}
\usepackage[utf8]{inputenc}
\usepackage{amssymb,xcolor}
\definecolor{darkred}{RGB}{180,0,0}
\definecolor{darkblue}{RGB}{0,0,180}
\usepackage{hyperref}
\hypersetup{
  colorlinks = true,
  urlcolor   = darkblue,
  linkcolor  = darkblue,
  citecolor  = darkred,
  pdftitle = {Dynamic asymptotic dimension for actions of virtually cyclic groups},
  pdfauthor={Massoud Amini, Kang Li, Damian Sawicki, and Ali Shakibazadeh}
}
\usepackage[capitalise]{cleveref}

\newtheorem{theorem}{Theorem}[section]
\newtheorem{cor}[theorem]{Corollary}
\newtheorem{prop}[theorem]{Proposition}
\newtheorem{lemma}[theorem]{Lemma}
\newtheorem{question}[theorem]{Question}
\newtheorem{rem}[theorem]{Remark}
\theoremstyle{remark}
 \newtheorem{remark}[theorem]{Remark}
\theoremstyle{definition}
\newtheorem{definition}[theorem]{Definition}

\newcommand\set[1]{\left\{#1\right\}}
\newcommand\range[2]{#1,\ldots,#2}
\newcommand{\acts}{\ensuremath\mathrel{\raisebox{.6pt}{$\curvearrowright$}}}
\newcommand*{\defeq}{\mathrel{\vcenter{\baselineskip0.5ex \lineskiplimit0pt
			\hbox{\scriptsize.}\hbox{\scriptsize.}}}%
	=}
\newcommand*{\eqdef}{=\mathrel{\vcenter{\baselineskip0.5ex \lineskiplimit0pt
			\hbox{\scriptsize.}\hbox{\scriptsize.}}}}

\def\NN{{\mathbb N}}
\def\ZZ{{\mathbb Z}}
\def\RR{{\mathbb R}}
\def\st{\ |\ }
\DeclareMathOperator{\dad}{dad}
\DeclareMathOperator{\dimnuc}{dim_{nuc}}

\title[DAD for actions of virtually cyclic groups]{Dynamic asymptotic dimension \\ for actions of virtually cyclic groups}

\author{Massoud Amini}
\address{Department of Mathematics, Faculty of Mathematical Sciences, Tarbiat Modares University, Tehran 14115-134, Iran}
\email{mamini@modares.ac.ir}

\author{Kang Li}
\address{Institute of Mathematics of the Polish Academy of Sciences, Śniadeckich 8, 00-656 Warsaw, Poland}
\email{kli@impan.pl}
\thanks{KL has received funding from the European Research Council (ERC) under the European Union's Horizon 2020 research and innovation programme (grant agreement no. 677120-INDEX)}

\author{Damian Sawicki}
\address{Max-Planck-Institut f\"{u}r Mathematik, Vivatsgasse 7,	53111 Bonn,	Germany}
\curraddr{KU Leuven, Department of Mathematics, Celestijnenlaan 200b – box 2400, B-3001 Leuven, Belgium}
\urladdr{sites.google.com/view/damiansawicki}
\thanks{DS thanks the Max Planck Institute for Mathematics in Bonn for its hospitality and the Max Planck Society for financial support.}

\author{Ali Shakibazadeh}
\address{Department of Mathematics, Faculty of Mathematical Sciences, Tarbiat Modares University, Tehran 14115-134, Iran}
\email{a.shakibazadeh@modares.ac.ir}

\begin{document}
\begin{abstract}
We show that the dynamic asymptotic dimension of an action of an infinite virtually cyclic group on a compact Hausdorff space is always one if the action has the marker property. This particularly covers a well-known result of Guentner, Willett, and Yu for minimal free actions of infinite cyclic groups. As  a direct  consequence,  we substantially extend  a  famous  result by  Toms  and  Winter  on  the  nuclear  dimension  of $C^*$-algebras  arising  from minimal free $\mathbb{Z}$-actions.

Moreover, we also prove the marker property for all free actions of countable groups on finite dimensional compact Hausdorff spaces, generalising a result of Szabó in the metrisable setting.
\end{abstract}

\maketitle

\section{Introduction}
Dynamic asymptotic dimension was introduced by Erik Guentner, Rufus Willett, and Guoliang Yu in \cite{gwy}. This is a notion of dimension for
actions of discrete groups on locally compact spaces, and it was also defined in the (more general) setting of locally
compact \'{e}tale groupoids. It is related to transfer reducibility of Bartels, L\"{u}ck, and Reich \cite{blr} and asymptotic dimension of Gromov \cite{gr}, and it can be used  \cite{ds} to bound the corresponding nuclear dimension of Winter and Zacharias \cite{wz} or to prove \cite{gwy2} instances of the Baum--Connes conjecture.

The paper \cite{gwy} calculates the exact value of dynamic asymptotic dimension for two classes of examples: the action of an arbitrary countable group on its Stone--\v{C}ech compactification and an arbitrary minimal action of $\mathbb Z$. The result in the latter case is that unless the space acted upon is finite, the dynamic asymptotic dimension of the action is always one \cite[Theorem 3.1]{gwy}.
The proof follows closely the ideas of Ian Putnam in building
$AF$-algebras associated to minimal  $\mathbb Z$-actions on the Cantor set \cite{p}. Note that minimal $\mathbb Z$-actions on infinite spaces are automatically free.

The purpose of this article is to prove the same rigidity result for free actions of infinite virtually cyclic groups. More precisely, we obtain it for free actions of infinite virtually cyclic groups on finite-dimensional compact Hausdorff spaces and on arbitrary compact Hausdorff spaces in the case of minimal actions. In particular, we show that the  dynamic asymptotic dimension of minimal free actions of the infinite dihedral group on compact spaces is always one.

The class of virtually cyclic groups plays a central role in geometric group theory. For instance, it appears in the Farrell--Jones conjecture \cite{fj}, which in turn implies the Borel conjecture \cite{Borel}. The proofs of instances of the Farrell--Jones conjecture involve establishing the finiteness of certain relative versions of equivariant asymptotic dimension (also known as the finite $\mathcal F$-amenability), a notion closely related to the dynamic asymptotic dimension.

Calculating the value of dynamic asymptotic dimension has so far proved to be rather difficult (cf.\ a question of Willett in \cite[Question 8.9]{piecewise}). A significant progress was made in \cite{SWZ}, where an exponential upper bound was obtained for free actions of nilpotent groups (on compact metric spaces of finite covering dimension), which was later improved to a linear estimate in \cite{piecewise}. These bounds depend on the dimension of the space acted upon. Already in the context of establishing the finiteness of dynamic asymptotic dimension rather than calculating its exact value, the relation between an action and its restriction to a finite index subgroup is not clear, and a different argument was used in order to extend the finiteness result of \cite{SWZ} to the class of \emph{virtually} nilpotent groups \cite{Bartels}. For actions on zero-dimensional spaces, the dynamic asymptotic dimension is known to be equal to certain other notions of dimension \cite{Kerr}, including the amenability dimension.

The finiteness of dynamic asymptotic dimension conjecturally implies finite dynamical complexity \cite{gwy2}, but at present this is only known if the dynamic asymptotic dimension equals zero or one. Therefore, it follows from our main result that many free actions of infinite virtually cyclic groups have finite dynamical complexity.

\section{Definitions and Results}

In the sequel, $\Gamma$ will denote a discrete group acting on a compact Hausdorff space $X$ by homeomorphisms. We shortly denote such an action by $\Gamma\acts X$. For a finite subset $E$ of $\Gamma$ we use the notation $E\Subset \Gamma$, and we denote the identity element of $\Gamma$ by $e$. When $A\subseteq \Gamma$ and $Y\subseteq X$, by $AY$ we will denote the union of translates of $Y$ by elements of $A$, namely $\bigcup_{g\in A} g Y = \{gy \st g\in A,\ y\in Y\}$.

Guentner,  Willett, and Yu define $\Gamma \acts X$ to have dynamic asymptotic dimension at most $d$ if for any $E\Subset \Gamma$ the space $X$ can be divided into $d + 1$
pieces in such a way that on each piece the ``action'' has ``finite complexity'' with respect to $E$ \cite{gwy}. More precisely, we have the following definition.

\begin{definition}\label{dad} For $E\subseteq \Gamma$ and an open subset $U\subseteq X$, the equivalence relation $\sim_{U,E}$ on $U$ generated by $E$ is defined as follows:  for $x, y \in U$,
  $x \sim_{U,E} y$ if there is $n\in\NN$ and a finite sequence
  $x = x_0, x_1,\ldots, x_n = y$
  in $U$ such that for each $1\leq j\leq n$, there exists $g\in E \cup E^{-1}$ such
  that $g x_{j-1} = x_j$.

  The \emph{dynamic asymptotic dimension} of a free action $\Gamma \acts X$ (denoted $\dad(\Gamma \acts X)$) is the
  smallest integer $d\in \NN$ with the following property: for each finite subset $E\Subset \Gamma$, there is an open cover $\{U_0,\ldots, U_d\}$ of $X$ such that for each $0\leq i\leq d$ the equivalence relation $\sim_{U_i,E}$ on $U_i$ has uniformly finite equivalence classes. If no such $d$ exists, we say that the dimension is infinite.
\end{definition}

When $U$ is the whole of $X$, the relation $\sim_{U,E} $ is the equivalence relation of being in the same $\langle E\rangle$-orbit, where $\langle E\rangle$ denotes the subgroup generated by $E$. If $\dad(\Gamma \acts X) = 0$, then there is no choice but to take $U_0=X$, and hence $\dad(\Gamma \acts X) = 0$ implies that $\langle E\rangle$-orbits are finite for every $E\Subset \Gamma$, i.e.\ $\Gamma$ is locally finite. However, if $U_i$ is a proper subset of $X$, then equivalence classes of  $\sim_{U_i,E}$ may be smaller than the intersection of $U_i$ with $\langle E\rangle$-orbits, and \cref{dad} requires a uniform bound on their cardinality (depending only on $E$).

Recall that an action $\Gamma \acts X$ is \emph{free} if only the identity element $e$ fixes a point in $X$. When this is not the case, in the above definition the uniform finiteness of equivalence classes has to be replaced by the finiteness of the following sets for $0\leq i\leq d$:
\begin{equation*}
    F(U_i, E) = \left\{\; g\in \Gamma \;\;\bigg|\;\;
    \begin{aligned}
      &\text{there exist } g_1,\ldots,g_n\in E \text{ and } x\in U_i \text{ such that } \\
      &g=g_n\cdots g_2g_1 \text{ and } g_j\cdots g_1x\in U_i \text{ for all } j\in \{1,\ldots,n\}
    \end{aligned}
    \;\right\}
  \end{equation*}
 (in this formulation, it is perhaps most natural to consider only symmetric $E\Subset\Gamma$).

\smallskip An action $\Gamma \acts X$ is \emph{minimal} if $X$ has no proper and non-empty closed (equivalently, open) $\Gamma$-invariant subsets. Minimality is the same as requiring that all $\Gamma$-orbits are dense in $X$. By the Ku\-ra\-tow\-ski--Zorn lemma, every compact $\Gamma$-space contains a closed $\Gamma$-invariant subset on which $\Gamma$ acts minimally.

\smallskip A discrete group $\Gamma$ is called \emph{virtually cyclic} (or \emph{cyclic by finite}) if it has a cyclic subgroup of finite index. When $\Gamma$ is infinite, this means that there is a finite index copy of $\mathbb Z$ inside~$\Gamma$. It follows that infinite virtually cyclic groups are residually finite and finitely presented.

An infinite virtually cyclic group $\Gamma$ is known to have a finite normal subgroup $N$ such that $\Gamma/N$ is either infinite cyclic or infinite dihedral (see e.g.\ Theorem~6.12 in Chapter~IV of \cite{dd}); depending on this one sometimes classifies $\Gamma$ into \emph{type I} or \emph{type II}. Moreover, in an infinite virtually cyclic group $\Gamma$, every infinite cyclic subgroup of $\Gamma$ has finite index. Indeed, if $H, K$ are two infinite cyclic subgroups of $\Gamma$ and $H$ is of finite index, then $H\cap K$ must be an infinite subgroup of $H$, so $[H: H\cap K]<\infty$. Since $H$ is of finite index, we get $[\Gamma: H\cap K]<\infty$, and hence $[\Gamma:K]<\infty$.

\smallskip The main result of this paper extends Theorem~3.1 of  \cite{gwy} from the case of infinite cyclic to that of infinite virtually cyclic groups and allows one to drop the minimality assumption.

\begin{theorem}\label{main}
  Let $ \Gamma \acts X $ be an action of an infinite virtually cyclic group on a compact Hausdorff space. Then
  \[\dad (\Gamma \acts X)= 1\]
  if $ \Gamma \acts X $ has the marker property, e.g.\ when it is free, and
  \begin{enumerate}
    \item $ \Gamma \acts X $ is minimal, or
    \item\label{fin-dim} $X$ has finite covering dimension.
  \end{enumerate}
\end{theorem}

This type of rigidity result has not been known for any concrete example of virtually cyclic groups except for the group of integers. In particular, \cref{main} applies to free minimal actions of the infinite dihedral group.

\Cref{fin-dim} of Theorem~\ref{main} follows from the following result of independent interest (for the definition of marker property, see Definition~\ref{marker-def}).

\begin{prop}\label{no-metrisability} Let $\Gamma \acts X$ be a free action of a countable group $\Gamma$ on a compact Hausdorff space $X$. If $\dim X$ is finite, then the action has the marker property.
\end{prop}

For $\ZZ$-actions, removing the minimality assumption from the result of Guentner, Willett, and Yu gives the following dimension reduction phenomenon.

\begin{rem}\label{finiteness-to-one} Let $\ZZ$ act by homeomorphisms on a compact Hausdorff space $X$ with $\dim X < \infty$. Then
$\dad(\ZZ\acts X)<\infty \implies \dad(\ZZ\acts X) = 1.$
\end{rem}
\begin{proof}
It is easy to see that the finiteness of $\dad$ implies that point stabilisers must be locally finite, but since $\ZZ$ is torsion-free, it implies in our case that the action is free. Now apply Theorem~\ref{main}.
\end{proof}

As already mentioned in the introduction, the notion of dynamic asymptotic dimension provides a way to compute an upper bound on the nuclear dimension of C$^*$-algebras arising from topological dynamical systems. The following corollary generalises \cite[Corollary~8.23]{gwy} to the setting considered here and additionally removes the metrisability assumption following the ideas of \cite[Corollary~5.4]{hw}.

\begin{cor}\label{nuclear for virtually cyclic}
Let $X$ be a compact Hausdorff space. If $\Gamma$ is a virtually cyclic group acting freely on $X$, then the nuclear dimension satisfies
\[\dimnuc(C(X)\rtimes_\mathrm{r}\Gamma)\leq 2\cdot \dim X +1,\]
where $C(X)\rtimes_\mathrm{r}\Gamma$ denotes the associated reduced crossed product $C^*$-algebra.
\end{cor}

\begin{remark}
It is worth noting that Corollary~\ref{nuclear for virtually cyclic} generalises the main technical result of \cite{TW}, Theorem~0.3, in two directions: from infinite cyclic groups to virtually cyclic groups and from minimal free actions to free actions. Additionally, it drops the metrisability assumption.

In the case of actions of $\ZZ$, the main result of \cite{hw} removes the freeness and metrisability assumption from \cite{TW} (at the price of replacing the linear bound from \cite{TW} and our Corollary~\ref{nuclear for virtually cyclic} with a quadratic upper bound).
\end{remark}

\section{Proofs}

\subsection{Marker property}
The following notion was introduced by Downarowicz \cite{Downarowicz} and in a modified form by Gutman \cite{Gutman}. Here we adapt Gutman's definition.  

\begin{definition}\label{marker-def} Let $\Gamma\acts X$ be a topological dynamical system and $F\Subset \Gamma$ a finite subset. We call an open set $U\subset X$ an \emph{$F$-marker}, if
\begin{itemize}
\item the family of sets $\set{g U : g \in F}$ is pairwise disjoint,
\item $X = \Gamma U$, i.e.\ translates of $U$ cover $X$.
\end{itemize}
We say that $\Gamma\acts X$ has the \emph{marker property} if there exist $F$-markers for all $F\Subset \Gamma$.
\end{definition}

For $\mathbb Z$-actions, the marker property is shown to be equivalent
to  strong  topological Rokhlin property, and systems with  marker property admit a compatible metric with
respect to which the metric mean dimension equals the (topological) mean
dimension \cite{Gutman2}. It is straightforward that the marker property holds for minimal free actions (see the following \cref{lemma1}), and  Szab\'o proved it for free actions on finite-dimensional metrisable spaces \cite{Szabo}. Proposition \ref{no-metrisability} removes the metrisability assumption.

\begin{lemma}\label{lemma1}
  Let $ \Gamma \acts X $ be a minimal free action by homeomorphisms of a group $ \Gamma $ on a compact Hausdorff space $X$. Then $ \Gamma \acts X $ has the marker property.
\end{lemma}

\begin{proof}
  Let $F\Subset \Gamma$ be given. We will find $U$ such that $U \cap h U = \emptyset$ for every $h\in E\defeq F^{-1}F \setminus \set e$, and hence $F$-translates of $U$ are disjoint.
  The claim is trivial if $X$ is empty, so pick $ x\in X $. By freeness, we have $ x\neq gx $ for every $ g\in E $. By Hausdoffness, for every $ g\in E $ there exist disjoint open neighbourhoods $ V_g$ and $ W_g $ of the points $ x$ and $ gx $. We define
  $ U= \bigcap_{g\in E}V_g\, \cap\, \bigcap_{g\in E}g^{-1}W_g $. Each set $ g^{-1}W_g $ is open because $ \Gamma $ acts by homeomorphisms, so $ U $ is open as a finite intersection of open sets, and it is also non-empty since it contains $ x $. Now, for any $ h\in E $ we have:
  \begin{align*}
    U\cap h U & =\Big(\underset{g\in E}{\bigcap}V_g \cap \underset{g\in E}{\bigcap}g^{-1}W_g \Big) \cap h\Big(\underset{g\in E}{\bigcap}V_g \cap \underset{g\in E}{\bigcap}g^{-1}W_g \Big)\\
    & = \Big(\underset{g\in E}{\bigcap} V_g \cap \underset{g\in E}{\bigcap}g^{-1} W_g \Big) \cap \Big(\underset{g\in E}{\bigcap}h V_g \cap \underset{g\in E}{\bigcap}hg^{-1} W_g \Big)\\
    & \subseteq V_h \cap hh^{-1} W_h = \emptyset.
  \end{align*}
  because $ V_h, W_h $ were chosen to be disjoint. Finally, since $U$ is non-empty and the action is minimal, we must have $\Gamma U = X$.
\end{proof}

To prove Proposition~\ref{no-metrisability} we will need the following.

\begin{lemma}\label{freeness} An action $\Gamma\acts X$ of a group $\Gamma$ on a compact Hausdorff space $X$ is free if and only if for every $\gamma\in \Gamma\setminus\set e$ there is a partition of unity $\set{\phi_i \in C(X) \st i\in I}$ indexed by a finite set $I$ such that $\gamma \phi_i \cdot \phi_i = 0$ for all $i\in I$.

In particular, if $\Gamma$ is countable and $\Gamma\acts X$ is a free action on a compact Hausdorff space $X$, then there exists a countable set $S \subseteq C(X)$ such that if $Y$ is a factor of $X$ satisfying $S\subseteq C(Y)$, then $\Gamma\acts Y$ is free.
\end{lemma}

When $Y$ is a factor of $X$, there is a quotient map $p\colon X\to Y$, and hence in the above we can consider maps $\phi\in C(Y)$ as maps on $X$ by pre-composition with $p$.

\begin{proof}[Proof of Lemma~\ref{freeness}]
The freeness of an action $\Gamma\acts X$ is equivalent to the fact that for every $\gamma\in \Gamma\setminus\set e$ and every $x\in X$ we have $\gamma x \neq x$. When this is the case, by Hausdorffness there are disjoint open sets $U$ and $V$ with $x\in U$ and $\gamma x \in V$, and putting $U_{x}^\gamma \defeq U\cap \gamma^{-1} V$ we obtain a neighbourhood of $x$ satisfying $\gamma U_{x}^\gamma \cap U_{x}^\gamma = \emptyset$. The family $\set{U_{x}^\gamma\st {x\in X}}$ is an open covering of $X$, and by compactness there is a finite subset $I_\gamma\subseteq X$ such that already the subcover $\set{U_{x}^\gamma\st x\in I_\gamma}$ covers $X$. Let $\set{\phi_x^\gamma \st x\in I_\gamma}$ be a partition of unity subordinate to this cover. Then the fact that $\gamma U_{x}^\gamma \cap U_{x}^\gamma = \emptyset$ immediately implies $\gamma \phi_x^\gamma \cdot \phi_x^\gamma = 0$.

Conversely, assume that $\set{\phi_i \in C(X) \st i\in I}$ is a partition of unity satisfying $\gamma \phi_i \cdot \phi_i = 0$ for some $\gamma\in \Gamma\setminus\set e$. Then the family $\set{U_i \st i\in I}$ given by $U_i = (\phi_i)^{-1}((0,\infty))$ is a covering of $X$ satisfying $\gamma U_i \cap U_i = \emptyset$. This means that for every $x\in X$ there is $U_i$ containing $x$, and then necessarily $\gamma x \neq x$ because $\gamma x$ and $x$ belong to the disjoint sets $\gamma U_i$ and $U_i$.

The last assertion of the lemma follows by taking
\[S= \set{\phi_i^\gamma \st \gamma\in \Gamma\setminus\set e,\, i\in I_\gamma}.\qedhere\]
\end{proof}

\begin{proof}[Proof of Proposition~\ref{no-metrisability}]
By Lemma~\ref{freeness}, there is a countable set $S$ such that if $Y$ is a factor of $X$ satisfying $S\subseteq C(Y)$, then $\Gamma\acts Y$ is free. In the special case of commutative C$^*$-algebras, Lemma 1.3 of \cite{hw} asserts that for such a set $S \subseteq C(X)$ there is a metrisable factor $Y$ of $X$ such that $S\subseteq C(Y)$ and $\dim Y \leq \dim X$. By \cite[Theorem~3.8]{Szabo}, freeness and finite-dimensionality imply that $\Gamma \acts Y$ satisfies the \emph{bounded topological small boundary property with respect to $d=\dim Y$}, which in turn (together with freeness) implies the marker property for $\Gamma \acts Y$ by \cite[Theorem~4.6]{Szabo}.

Finally, it is immediate that if $V\subseteq Y$ is an open $F$-marker for some finite set $F\Subset \Gamma$ and the action $\Gamma\acts Y$, then $U = p^{-1}(V)$ is an open $F$-marker for  the action $\Gamma\acts X$, where $p$ denotes the factor map $X\to Y$. Thus, the marker property for $Y$ yields the marker property for $X$.
\end{proof}

\subsection{The proof of the main theorem}

\begin{proof}[Proof of \cref{main}]
  The dynamic asymptotic dimension equals $0$ only for actions of locally finite groups, and $\Gamma$ contains an infinite-order element, so $\operatorname{dad}(\Gamma\acts X) \geq 1$. Hence, it suffices to prove the reverse inequality.

  It is well known that an infinite virtually cyclic group $\Gamma$ has a maximal finite normal subgroup $H\vartriangleleft\Gamma$, and the quotient $\Gamma/H$ is either $\ZZ $ or $D_\infty = \langle s,t : s^2=t^2=e \rangle$. We will treat both cases simultaneously. Let $p$ denote the quotient map $\Gamma\to \Gamma/H$.

  To obtain the upper bound, we will use a covering very similar to the one in \cite{gwy}, but the corresponding reasoning has to be different in order to apply to type II virtually cyclic groups. Let a finite $E\Subset \Gamma$ be given. By enlarging it if necessary, one can assume that $E=p^{-1}(B_N)$ for some $N\in \NN_{>0}$, where $B_N$ denotes the ball of radius $N$ around the identity element in the group $\Gamma/H$ with respect to the standard metric on $\ZZ$ or the right-invariant word metric on $D_\infty$ associated with the generating set $\{s,t\}$.

  By the marker property, there exists an open subset $U\subseteq X$ such that $U\cap gU=\emptyset$ for $g\in p^{-1}(B_{5N})\setminus \set{e}$  and  $\Gamma U = X$. By compactness, there is a finite set $F \Subset \Gamma$ such that already $\set{g U \st g\in F}$ covers $X$. By normality, one can find sets $U^g$ for $g\in F$ such that $\overline {U^g} \subseteq g U$ and such that $\set{U^g \st g\in F}$ still covers $X$. The set $V \defeq \bigcup_{g\in F} g^{-1} U_g$ satisfies $\overline V \subseteq U$ and $F V = X$. Define
  \[U_0 =  p^{-1}(B_N) U \quad \mathrm{and} \quad U_1 = X\setminus p^{-1}(B_N) \overline{V}.\]
  Clearly, $\{U_0,U_1 \}$ forms an open cover of $X$.

  It now suffices to prove that for $i\in\set{0,1}$ the following set is finite
  \begin{equation*}
    F(U_i, E) = \left\{\; g\in \Gamma \;\;\bigg|\;\;
    \begin{aligned}
      &\text{there exist } g_1,\ldots,g_n\in E \text{ and } x\in U_i \text{ such that } \\
      &g=g_n\cdots g_2g_1 \text{ and } g_j\cdots g_1x\in U_i \text{ for all } j\in \{1,\ldots,n\}
    \end{aligned}
    \;\right\}.
  \end{equation*}

  We begin with $i=0$, claiming that $p(F(U_0, E))\subseteq B_{3N}$. Suppose for contradiction that there exist $ g_1, \ldots,  g_n \in E$ and $x\in U_0$ such that $p( g_n\cdots g_1)\notin B_{3N}$, and $x_j \defeq g_j\cdots g_1 x\in U_0$ for all $j\in \set{\range{0}{n}}$ (for $j=0$ one just gets $x_0 \defeq x$). There exists $k \in \set{\range{3}{n}}$ such that $p( g_k\cdots g_1)\in B_{3N}\setminus B_{2N}$. By the definition of $U_0$, the point $x$ can be expressed as $g x_U$ for some $g\in p^{-1}(B_N)$ and $x_U\in U$. Similarly, $x_k = g' x_U'$ for some $g'\in p^{-1}(B_N)$ and $x_U'\in U$. But then
  \[x_U' = (g')^{-1} x_k = (g')^{-1} g_k\cdots g_1 x = (g')^{-1} g_k\cdots g_1 g x_U,\]
  which is a contradiction, because
  $h \defeq (g')^{-1} g_k\cdots g_1 g$
  belongs to $p^{-1}(B_{5N} \setminus B_0) \subseteq p^{-1}(B_{5N}) \setminus \set{e}$,
  and we assumed that $U\cap h U = \emptyset$ for such $h$.

  \medskip We are done with $F(U_0, E)$, so let us now consider $F(U_1, E)$. Since $\Gamma V = X$, by compactness there is $M\in \NN$ such that already $p^{-1}(B_M) V$ equals $X$. We claim that $p(F(U_1, E))\subseteq B_{2M+N}$.

  As before, suppose for contradiction that there exist $ g_1, \ldots,  g_n \in E$ and $x\in U_1$ such that $g_n\cdots g_1\notin p^{-1}(B_{2M+N})$, and $x_j \defeq g_j\cdots g_1 x\in U_1$ for all $j\in \set{\range{0}{n}}$.

  Note that $\Gamma / H$ is isometric to $\ZZ \subseteq \RR$. Let $\iota$ denote the isometry $\Gamma / H \to \ZZ$ that maps the neutral element of $\Gamma / H$ to $0\in \ZZ$ and such that $\iota\circ p(g_n\cdots g_1) \eqdef L > 0$. Since $p(g_k)\in B_N$ for all $k\in \set{\range{1}{n}}$, there exists such $k$ that $\iota\circ p(g_k\cdots g_1) \eqdef l \in [M+1,M+N]$. Note that $| g_n\cdots g_{k+1} | = d(g_n\cdots g_1, g_k\cdots g_1) = L-l$.

  Since $p^{-1}(B_M) V = X$, there is $g\in p^{-1}(B_M)$ such that $ g^{-1} x_k \in V$. By multiplying the sequence $(g_j\cdots g_1)_{j=0}^n$ by $g_1^{-1} \cdots g_k^{-1} g$ on the right, we obtain the following sequence $(h_j)_{j=0}^n$:
  \[g_1^{-1} \cdots g_k^{-1}g,\;\;g_2^{-1} \cdots g_k^{-1}g,\;\;\ldots,\;\;g_k^{-1}g,\;\;g,\;\;g_{k+1} g,\;\;\ldots,\;\;g_n \cdots g_{k+1} g. \]
  Note that the length of $p(h_0)$ is at most $l+M\leq 2M+N < L$, and the length of $p(h_n)$ is at most $L - l + M < L$. Since the distance between $p(h_0)$ and $p(h_n)$ equals $L$, their images under $\iota$ must have opposite signs. Consequently, there exists $j\in \set{\range{0}{n}}$ such that $p(h_j)\in B_N$. This is a contradiction: on the one hand, we assumed that $g_j\cdots g_1 x \in U_1 = X\setminus p^{-1}(B_N) \overline V$, and on the other hand
  \[g_j\cdots g_1 x = (g_j\cdots g_1)(g_1^{-1} \cdots g_k^{-1} g)(g^{-1} g_k \cdots g_1) x = h_j g^{-1} x_k \in p^{-1}(B_N) V. \qedhere \]
\end{proof}

\vspace{.3 cm}

It is not difficult to see that if $K$ is a finite normal subgroup of $\Gamma$, and $\Gamma$ acts on $X$, then
\begin{equation}\label{quotient-by-finite}
 \dad(\Gamma\acts X) \leq \dad(\Gamma/K \acts X/K).
\end{equation}
Indeed, let us denote $q\colon \Gamma\to \Gamma/K$ and $p\colon X\to X/K$. If $E \Subset \Gamma$ and $\{\range{U_0}{U_d}\}$ is a covering of $X/K$ such that the sets $F(U_i, q(E))$ are finite for $i\in\set{\range{0}{d}}$, then $\{\range{p^{-1}(U_0)}{p^{-1}(U_d)}\}$ is a covering of $X$, and the sets $F(p^{-1}(U_i), E)$ are finite as contained in the respective $q^{-1}(F(U_i, q(E)))$.

Using inequality \eqref{quotient-by-finite}, one can slightly generalise \cref{main} to cover some non-free actions, as it may happen that $\Gamma/K \acts X/K$ is free even if $\Gamma \acts X$ is not, while the remaining assumptions of minimality and finite covering dimension are preserved when passing from $\Gamma \acts X$ to $\Gamma/K \acts X/K$. Inequality \eqref{quotient-by-finite} could also be used to give an slight extension of  \cref{finiteness-to-one} for the case of the action of a virtually cyclic group of \emph{type I} (by reducing the problem to that of a $\mathbb Z$-action). However we do not know at this stage if  \cref{finiteness-to-one} holds for the action of a virtually cyclic group of \emph{type II}.

\subsection{Consequences for nuclear dimension}

\begin{proof}[Proof of Corollary~\ref{nuclear for virtually cyclic}]
If $\dim X=\infty$, we have nothing to prove. So we only have to concentrate on the case where $\dim X$ is finite.

Assume first that $X$ is metrisable. Then Theorem~8.6 of \cite{gwy} (together with their Lemma~5.4 and Remark~8.5 (i)) asserts that
\[\dimnuc(C(X)\rtimes_\mathrm{r}\Gamma) \leq (\dad(\Gamma\acts X) + 1)(\dim X + 1) - 1,\]
and we have $\dad(\Gamma\acts X) \leq 1$ by Theorem~\ref{main} ($\dad (\Gamma \acts X)$ equals $1$ when $\Gamma$ is infinite, and it vanishes by definition when $\Gamma$ is finite).

Now, we consider the general case. By Lemma~\ref{freeness}, there is a countable subset $A \subseteq C(X)$ such that when $Y$ is a factor of $X$ satisfying $C(Y) \supseteq A$, the action on $Y$ is free. In the special case of commutative C$^*$-algebras, Lemma 1.3 of \cite{hw} asserts that for any countable subset $B \subseteq C(X)$ there is a metrisable factor $Y_B$ of $X$ such that $B\subseteq C(Y_B)$ and $\dim Y_B\leq \dim X$. If we assume additionally $A\subseteq B$, then the already proven part gives $\dimnuc(C(Y_B)\rtimes_\mathrm{r}\Gamma)\leq 2\cdot \dim X +1$. Clearly, the union of all such C$^*$-algebras $C(Y_B)\rtimes_\mathrm{r}\Gamma$ is $C(X)\rtimes_\mathrm{r}\Gamma$, and they form a direct system under inclusion, so we conclude that
\[\dimnuc \left(C(X)\rtimes_\mathrm{r}\Gamma\right) \leq \liminf_{Y_B} \dimnuc \left(C(Y_B)\rtimes_\mathrm{r}\Gamma\right) \leq 2\cdot \dim X +1\]
by the fact that nuclear dimension behaves well under direct limits (\cite[Proposition~2.3~(iii)]{wz}).
\end{proof}

We end this note by raising the following question inspired by Theorem~\ref{main}: 
\begin{question}
Does finite dynamic asymptotic dimension of an action of $\mathbb{Z}$ on a compact metric space imply the marker property of the action? 
\end{question}
In the above case, the action must be free and it is an open problem whether freeness implies the marker property in general (see \cite[Problem~5.4]{Gutman}).


\begin{thebibliography}{00}

  \bibitem{Bartels}
  A. Bartels, \textit{Coarse flow spaces for relatively hyperbolic groups}, Compos. Math. \textbf{153} (2017), no. 4, 745--779.

  \bibitem{Borel}
  A. Bartels and W. L\"{u}ck, \textit{The Borel conjecture for hyperbolic and ${\rm CAT}(0)$-groups}, Ann. of Math. (2) \textbf{175} (2012), no. 2, 631--689.

  \bibitem{blr}
  A. Bartels, W. L\"{u}ck, and H. Reich, \textit{Equivariant covers for hyperbolic groups}, Geom. Topol. \textbf{12} (2008), 1799--1882.

\bibitem{Downarowicz}
T. Downarowicz, \textit{Minimal models for noninvertible and not uniquely ergodic
systems}, Israel J. Math. \textbf{12} (2006), 156:93--110.

  \bibitem{ds}
  R. J. Deeley and K. R. Strung, \textit{Nuclear dimension and classification of $\mathrm{C}^*$-algebras associated to Smale spaces}, Trans. Amer. Math. Soc. \textbf{370} (2018), 3467--3485.

  \bibitem{dd}
  W. Dicks and M. Dunwoody, \textit{Groups acting on graphs}, Cambridge Studies in Advanced Mathematics \textbf{17}, Cambridge University Press, Cambridge, 1989.

  \bibitem{fj}
  F. T. Farrell and L. E. Jones, \textit{Isomorphism conjectures in algebraic K-theory}, J. Amer. Math. Soc. \textbf{6} (1993), 249--297.

  \bibitem{gr} M. Gromov, Asymptotic invariants of infinite groups. In: Geometric group theory, G. Niblo and  M. Roller (eds.), vol. 2, London Mathematical Society, London, 1993.

  \bibitem{gwy2} E. Guentner, R. Willett, and G. Yu,  \textit{Dynamical complexity and controlled operator $K$-theory}, arXiv:1609.02093 [math.KT].

  \bibitem{gwy}
  E. Guentner, R. Willett, and G. Yu, \textit{Dynamic asymptotic dimension: relation to dynamics, topology, coarse geometry, and $C^*$-algebras}, Math. Ann. \textbf{367} (2017), 785--829.

  \bibitem{Gutman}
   Y. Gutman, \textit{Mean dimension and Jaworski-type theorems}, Proc. Lond. Math. Soc. (3) \textbf{111} (2015) 831--850.
   
 \bibitem{Gutman2}
 Y. Gutman, \textit{Embedding topological dynamical systems with periodic points in cubical shifts}, Ergodic Theory Dynam. Systems \textbf{37} (2017), no. 2, 512--538.
 

   \bibitem{hw}
   I. Hirshberg and J. Wu, \textit{The nuclear dimension of $C^*$-algebras associated to homeomorphisms},  Adv. Math. \textbf{304} (2017), 56--89.

  \bibitem{Kerr}
  D. Kerr, \textit{Dimension, comparison, and almost finiteness}, J. Eur. Math. Soc. (JEMS) \textbf{22} (2020), no. 11, 3697--3745. 

  \bibitem{p} I. Putnam, \textit{The $C^*$-algebras associated with minimal homeomorphisms of the Cantor set}, Pac. J. Math. \textbf{136}(2) (1989), 329--353.

  \bibitem{piecewise}
  D. Sawicki, \textit{Warped cones, (non-)rigidity, and piecewise properties. With a joint appendix with Dawid Kielak}, Proc. Lond. Math. Soc. (3) \textbf{118} (2019), no. 4, 753--786.

  \bibitem{Szabo}
  G. Szabó, \textit{The Rokhlin dimension of topological $\mathbb Z^m$-actions},
Proc. Lond. Math. Soc. (3) \textbf{110} (2015), no. 3, 673--694.

  \bibitem{SWZ}
  G. Szab\'o, J. Wu, and J. Zacharias, \textit{Rokhlin dimension for actions of residually finite groups}, Ergodic Theory Dynam. Systems \textbf{39} (2019), no. 8, 2248--2304.

  \bibitem{TW}
  A. S. Toms and W. Winter,
 \textit{Minimal dynamics and K-theoretic rigidity: Elliott's conjecture}, Geom. Funct. Anal. \textbf{23} (2013), no. 1, 467--481.

  \bibitem{wz} W. Winter and J. Zacharias,  \textit{The nuclear dimension of $C^*$-algebras}, Adv. Math. \textbf{224} (2) (2010), 461--498.







\end{thebibliography}
\end{document}